\documentclass[10pt]{article}

\usepackage{amsmath}
\usepackage{amsfonts}
\usepackage{amsthm}
\usepackage[english]{babel}
\usepackage{graphicx}
\usepackage[all]{xy}
\usepackage{amssymb}
\usepackage{amsmath}
\usepackage{tikz}
\usetikzlibrary{arrows,positioning} 
\usetikzlibrary{shapes.geometric}

\newtheorem{theorem}{Theorem}
\newtheorem{lemma}{Lemma}
\newtheorem{definition}{Definition}

\newtheorem{remark}{Remark}

\newcommand{\scalar}[1]{\langle #1\rangle}
\newcommand{\bigscalar}[1]{\Big\langle #1\Big\rangle}
\newcommand{\floor}[1]{\lfloor #1\rfloor}
\newcommand{\diff}[1]{\partial_{#1}}
\newcommand{\diffc}[1]{\bar{\partial}_{#1}}
\newcommand{\norm}[1]{\Vert #1\Vert}
\newcommand{\abs}[1]{\lvert #1\rvert}

\newcommand{\mR}{\mathbb{R}}
\newcommand{\mC}{\mathbb{C}}
\newcommand{\mN}{\mathbb{N}}
\newcommand{\mE}{\mathbb{E}}

\newcommand{\mS}{\mathbb{S}}

\newcommand{\mH}{\mathbb{H}}

\newcommand{\cE}{\mathcal{E}}

\newcommand{\cH}{\mathcal{H}}
\newcommand{\cP}{\mathcal{P}}
\newcommand{\cR}{\mathcal{R}}

\newcommand{\ubar}{\bar{u}}
\newcommand{\zbar}{\bar{z}}
\newcommand{\sbar}{\bar{s}}
\newcommand{\tbar}{\bar{t}}

\newcommand*\col[3][]{\begin{pmatrix}\ifx\relax#1\relax\else#1\\\fi{#2}\\{#3}\end{pmatrix}}

\begin{document}
\title{Plane wave formulas for spherical, complex and symplectic harmonics}
\author{H.\ De Bie\footnote{E-mail: {\tt Hendrik.DeBie@UGent.be}} \and F. Sommen\footnote{E-mail: {\tt Frank.Sommen@UGent.be}} \and M. Wutzig\footnote{E-mail: {\tt Michael.Wutzig@UGent.be}}}
\vspace{10mm}
\date{\small{Department of Mathematical Analysis\\ Faculty of Engineering and Architecture -- Ghent University}\\
\small{Krijgslaan 281, 9000 Gent, Belgium}\\ \vspace{5mm}
}
\maketitle

\begin{abstract}
\noindent This paper is concerned with spherical harmonics, and two refinements thereof: complex harmonics and symplectic harmonics.\\
\indent The reproducing kernels of the spherical and complex harmonics are explicitly given in terms of Gegenbauer or Jacobi polynomials. In the first part of the paper we determine the reproducing kernel for the space of symplectic harmonics, which is again expressible as a Jacobi polynomial of a suitable argument.\\
\indent In the second part we find plane wave formulas for the reproducing kernels of the three types of harmonics, expressing them as suitable integrals over Stiefel manifolds. This is achieved using Pizzetti formulas that express the integrals in terms of differential operators.\\
\newline
\noindent{\it Keywords\/}: reproducing kernels, spherical harmonics, symplectic harmonics, Pizzetti formulas, plane waves, Stiefel manifolds\\
\newline
\noindent{\it Mathematics Subject Classification\/}: 32A50, 42B35
\end{abstract}

\section{Introduction}
In the study of harmonic analysis on the sphere (see e.g. \cite{xu}), spherical harmonics play a crucial role. Spherical harmonics of degree $k$ are (the restrictions to the unit sphere of) harmonic polynomials of homogeneity $k$. In many applications, one is not so much interested in the harmonics themselves, but rather in their reproducing kernels. For spherical harmonics, the reproducing kernel is given in terms of Gegenbauer polynomials (see \cite{stein} and subsequent Theorem \ref{Gegenbauer}).\\
\indent Several refinements of spherical harmonics have been introduced over the last decades. Koornwinder (see \cite{koorn}) defined complex harmonics of degree $(p,q)$ as spherical harmonics which are homogeneous of order $p$ in the complexified variables and of order $q$ in the complex conjugated variables. He determined the reproducing kernel of the complex harmonics in terms of Jacobi polynomials (see also subsequent Theorem \ref{kernelC}), which was the first step in establishing his celebrated addition formula for the Jacobi polynomials.\\
\indent Recently, a further refinement of complex harmonics was introduced in \cite{symp}. They are called symplectic harmonics, as the space they span is invariant under the symplectic group. In addition to being $(p,q)$ complex harmonics, symplectic harmonics have to be in the kernel of a certain twisted Euler operator (see the subsequent Definition \ref{operators}).\\
\indent Let us now describe the main results of our paper. First we determine the reproducing kernel of the space of symplectic harmonics, by constructing a suitable projection operator. Surprisingly, as in the complex case, the result is expressed in terms of a Jacobi polynomial but now of a more complicated argument.\\
\indent In the second part of the paper, we construct integral formulas for the reproducing kernels of the three types of harmonics. They express the reproducing kernels as integrals over real or complex Stiefel manifolds of much simpler integrands given by plane wave polynomials. To achieve this, we make use of the Pizzetti formulas established in \cite{stiefel} that express the Stiefel integral as a series expansion in suitable differential operators.\\
\indent In the case of spherical harmonics, this integral formula was obtained in a completely different way in \cite{radon}. Moreover, there also exists an alternative integral expression over a similar geometric object, due to Sherman, where the integrand is a rational function \cite{sherman} instead of a polynomial.\\
\indent The paper is organized as follows. In Section \ref{pre} we fix some notations and summarize important results on spherical and complex harmonics and their respective reproducing kernels. In Section \ref{symplectic} we give a brief introduction to symplectic polynomials and symplectic harmonics. A projection operator onto these spaces is constructed and used to determine reproducing kernels for both spaces. Finally, in Section \ref{planewaves} we establish integral formulas for the reproducing kernels for spherical, complex and symplectic harmonics.
\section{Preliminaries}
\label{pre}
\subsection{Spherical harmonics}
Let the space of complex-valued polynomials on $\mR^m$ that are homogeneous of degree $k$ be denoted by $\cP_k$. Using a multi-index $\alpha = [\alpha_1,\cdots,\alpha_m]\in\mN^m$ and the conventional notation $x^{\alpha}=x_1^{\alpha_1}x_2^{\alpha_2}\cdots x_m^{\alpha_m}$ such a polynomial can be written as 
$P(x)=\sum_{|\alpha|=k}c_\alpha x^\alpha,$ with coefficients $c_\alpha\in\mC$, where the sum runs over all possible index sets of length $|\alpha|=\alpha_1+\cdots+\alpha_m=k$. Clearly, $\cP_k$ is the eigenspace of the Euler operator $\mE=\sum_{j=1}^mx_j\diff{x_j}$ with eigenvalue $k$.\\
\indent On the space of complex-valued polynomials (of arbitrary degree) $\cP=\bigoplus_{j=0}^\infty \cP_k$ the usual inner product 
is the so-called Fischer inner product (see \cite{Fi})
\begin{align*}
	\bigscalar{P(x),Q(x)}_\partial=\Big[\overline{P(\partial)}Q(x)\Big]_{x=0},
\end{align*}
where $P(\partial)$ arises by substituting every $x_j$ by $\diff{x_j}$ in $P(x)$. It is easy to see that
with respect to this inner product the homogeneous polynomial (as a function in $x$)
\begin{align*}
	Z_k(x,y)=\frac{\scalar{x,y}^k}{k!},\text{ with }\scalar{x,y}=\sum_{j=1}^mx_jy_j
\end{align*}
acts as a reproducing kernel on $\cP_k$, i.e.
\begin{align*}
	\bigscalar{Z_k(x,y),P_\ell(x)}_\partial=\delta_{kl}P_\ell(y)
\end{align*}
for any homogeneous polynomial $P_\ell\in\cP_\ell$. One immediately sees that $x_j$ and $\diff{x_j}$ are dual with respect to $\scalar{\cdot,\cdot}_\partial$ and therefore the dual of $\norm{x}^2=\sum_{j=1}^mx_j^2$ is the Laplacian $\Delta=\sum_{j=1}^m\diff{x_j}^2$. We have indeed for two polynomials $P$ and $Q$ that
\begin{align}
\label{dual}
	\bigscalar{\norm{x}^2P(x),Q(x)}_\partial=\bigscalar{P(x),\Delta Q(x)}_\partial.
\end{align}
The space of spherical harmonics $\cH_k$ of degree $k$ consists of $k$-homogeneous polynomials, that are in the kernel of the Laplacian, hence
\begin{align}
\label{dim}
	\cH_k=\Big\{H(x):H(x)\in\cP_k,\Delta H(x)=0\Big\}\text{ with }\dim{\cH_k}=\frac{2k+m-2}{k+m-2}{k+m-2\choose{m-2}}.
\end{align}
As homogeneous polynomials they are of course orthogonal with respect to the Fischer inner product. Another inner product on
$\cH_k$ is the spherical $L^2$ inner product, i.e.
\begin{align*}
	\bigscalar{P(x),Q(x)}_{\mS}=\frac{1}{\omega_{m-1}}\int_{\mS^{m-1}}\overline{P(x)}Q(x)d\sigma(x),
\end{align*}
where $d\sigma(x)$ is the usual Lebesgue measure on the sphere $\mS^{m-1}$ with surface $\omega_{m-1}=2\pi^{\frac{m}{2}}/\Gamma(\frac{m}{2})$. A very
useful property of these two inner products is their proportionality, which was shown in e.g. \cite{xu}. We state this result here, as it plays a central role in our paper.
\begin{theorem}
\label{tprop}
	For a spherical harmonic $H_k\in\cH_k$ and a homogeneous polynomial $P_k\in\cP_k$ of the same degree it holds that
	\begin{align}
	\label{prop}
		2^k\Big(\frac{m}{2}\Big)_k\bigscalar{H_k(x),P_k(x)}_{\mS}=\bigscalar{H_k(x),P_k(x)}_\partial.
	\end{align}
	For two spherical harmonics of different degree $H_k\in\cH_k$ and $Q_l\in\cH_l$, with $k\neq l$, one has in addition that
	\begin{align*}
		2^k\Big(\frac{m}{2}\Big)_k\bigscalar{H_k(x),Q_l(x)}_{\mS}=\bigscalar{H_k(x),Q_l(x)}_\partial=0,
	\end{align*}
	where $(a)_k=a(a+1)\cdots(a+k-1)$ denotes the Pochhammer symbol.
\end{theorem}
A well-known fact is that homogeneous polynomials can be decomposed into spherical harmonics of lower degrees (see e.g. \cite{stein}). This Fischer decomposition is given in the following lemma.
\begin{lemma}
	For the space of homogeneous polynomials $\cP_k$ it holds that
	\begin{align*}
		\cP_k=\bigoplus_{j=0}^{\floor{\frac{k}{2}}}\norm{x}^{2j}\cH_{k-2j}.
	\end{align*}
\end{lemma}
\noindent Mapping a given $k$-homogeneous polynomial to its harmonic component of degree $k-2\ell$ can be done by means of the projection operator (see e.g. \cite{orst} and \cite{sob})
\begin{align}
	\label{oproj}
	Proj^ k_{\ell}=\sum_{j=0}^{\floor{\frac{k}{2}}-\ell}\alpha_j\norm{x}^{2j}\Delta^{j+\ell},
\end{align}
with $\alpha_j=\frac{(-1)^j(\frac{m}{2}+k-2\ell-1)}{4^{j+\ell}j!\ell!}\frac{\Gamma(\frac{m}{2}+k-2\ell-j-1)}{\Gamma(\frac{m}{2}+k-\ell)}$. By construction we have
\begin{align}
\label{proj}
	Proj^k_\ell\big(\norm{x}^{2j}H_{k-2j}\big)=\delta_{j\ell}H_{k-2j},
\end{align}
for the spherical harmonic $H_{k-2j}\in\cH_{k-2j}$. It is a classical result that the reproducing kernel on $\cH_k$ can be given in terms of a Gegenbauer polynomial, compare e.g. with \cite{stein} where it is called the zonal harmonic.
\begin{theorem}
\label{Gegenbauer}
	For a spherical harmonic $H_j\in\cH_j$ it holds that
	\begin{align*}
		\bigscalar{K_k(x,y),H_j(x)}_\mS=\delta_{jk}H_j(y),
	\end{align*}
	where
	\begin{align*}
		K_k(x,y)=\frac{k+\mu}{\mu}\norm{x}^k\norm{y}^kC_k^\mu(t),
	\end{align*}
	with $\mu=\frac{m}{2}-1$, $t=\frac{\scalar{x,y}}{\norm{x} \norm{y} }$ and $C_k^\mu(t)$ the Gegenbauer polynomial.
\end{theorem}
One way to derive this kernel is by projecting the homogeneous kernel $Z_k$ onto the harmonics of the same degree (see Theorem 1.2.6. in \cite{xu}). We mention this strategy here as we will use it again in Section \ref{symplectic}.
\begin{theorem}
\label{projkernel}
	The reproducing kernel $K_k(x,y)$ of spherical harmonics can be given as the harmonic projection $K_k(x,y)=c_k Proj^k_0\Big(Z_k(x,y)\Big)$ 
	of the homogeneous kernel $Z_k(x,y)=\frac{\scalar{x,y}^k}{k!}$, where $c_k=2^k\big(\frac{m}{2}\big)_k$.
\end{theorem}
\subsection{Complex harmonics}
When considering functions defined on an even-dimensional vector space one can apply a complex structure to identify $\mR^{2n}$ with
the complex vector space $\mC^n$. Vectors $z=[z_1,\cdots,z_n]\in\mC^n$ have the coordinates $z_j=x_j+ix_{n+j}$ and complex conjugation is denoted
by $\zbar_j=x_j-ix_{n+j}$. Complex derivatives are defined in the usual way as $\diff{z_j}=\frac{1}{2}(\diff{x_j}-i\diff{x_{n+j}})$ and $\diffc{z_j}=\frac{1}{2}(\diff{x_j}+i\diff{x_{n+j}})$. The Euler operator can now be split into two complex conjugated Euler operators as $\mE=\mE_z+\bar{\mE}_z$, where
\begin{align*}
	\mE_z=\sum_{j=1}^nz_j\diff{z_j}\hspace{1cm}\text{and}\hspace{1cm}\bar{\mE}_z=\sum_{j=1}^n\zbar_j\diffc{z_j}.
\end{align*}
Bihomogeneous polynomials of degree $(p,q)$ are eigenfunctions of these two operators with the respective eigenvalues, hence
\begin{align*}
	\cP_{p,q}=\{P(z,\bar{z}):\mE_zP(z,\bar{z})=pP(z,\bar{z}),\bar{\mE}_zP(z,\bar{z})=qP(z,\bar{z})\}.
\end{align*}
Note that from here on we will write $P(z)$, rather than $P(z,\zbar)$, to indicate that a function depends on both $z$ and $\zbar$, and hence on the vector $x$. The space of $k$-homogeneous polynomials can be split further into spaces of bihomogeneous polynomials, i.e. $\cP_k=\bigoplus_{j=0}^k\cP_{j,k-j}$. As in the real case, homogeneous polynomials of different (bi-)degrees are orthogonal with respect to a Fischer inner product that is defined as
\begin{align}
\label{fischer}
	\bigscalar{P(z),Q(z)}_{\partial}=\Big[\overline{P(\partial)}Q(z)\Big]_{z=0}.
\end{align}
Here $P(\partial)$ is obtained by substituting every $z_j$ by $\diffc{z_j}$ and every $\zbar_j$ by $\diff{z_j}$ in $P(z)$. This definition differs from the one in the real case by a factor of $2$. For the sake of readability we will nonetheless use the same notation in both cases. With respect to this complex Fischer inner product we have for $P_{r,s}\in\cP_{r,s}$ that
\begin{align*}
	\bigscalar{Z_{p,q}(z,u),P_{r,s}(z)}_\partial=\delta_{pr}\delta_{qs}P_{r,s}(u)
\end{align*}
with the reproducing kernel
\begin{align}
\label{zpq}
	Z_{p,q}(z,u)=\frac{\scalar{z,\ubar}^p\scalar{\zbar,u}^q}{p!q!},
\end{align}
for the complex vectors $z,u\in\mC^n$ and the Hermitian inner product $\scalar{z,\ubar}=\sum_{j=1}^nz_j\ubar_j$. Note that this definition is different from the standard one. We use it to avoid confusion in Section \ref{planewaves}.\\
\indent Spherical harmonics of bi-degree $(p,q)$ are defined as harmonic bihomogeneous polynomials of the same order (see e.g. \cite{stras}, \cite{sapiro}), hence
\begin{align}
\label{dimC}
	\cH_{p,q}=\Big\{H(z):H(z)\in\cP_{p,q},\Delta_zH(z)=0\Big\}\text{ with }\dim{\cH_{p,q}}=\frac{n+p+q-1}{n-1}{q+n-2\choose{n-2}}{p+n-2\choose{n-2}},
\end{align}
where $\Delta_z=\sum_{j=1}^{n}\diffc{z_j}\diff{z_j}$. Note that this definition of the Laplacian is proportional to the one given above, as $4\Delta_z=\Delta$. The spherical $L^2$ inner product on $\cH_{p,q}$ now reads as
\begin{align*}
	\bigscalar{P(z),Q(z)}_\mS=\frac{1}{\omega_{2n-1}}\int_{\mS^{2n-1}}\overline{P(z)}Q(z)d\sigma(z)
\end{align*}
and satisfies the same proportionality relation with the (complex) Fischer inner product, that was given in Theorem \ref{tprop}. We restate this relation for the current framework in the following lemma.
\begin{lemma}
\label{cprop}
	For a spherical harmonic $H_{p,q}\in\cH_{p,q}$ and a homogeneous polynomial $P_{p,q}\in\cP_{p,q}$ it holds, that
	\begin{align*}
		\big(n\big)_{p+q}\bigscalar{H_{p,q}(z),P_{p,q}(z)}_{\mS}=\bigscalar{H_{p,q}(z),P_{p,q}(z)}_\partial.
	\end{align*}
\end{lemma}
\noindent Note that the factor $2^k$ in equation (\ref{prop}) has been omitted here because of the definition of the complex Fischer inner product given in (\ref{fischer}). As before, spherical harmonics of different bi-degrees are orthogonal with respect to both inner products. We therefore have an orthogonal decomposition of bihomogeneous polynomials $\cP_{p,q}$ into harmonics of lower degrees.
\begin{lemma}
	For the space of homogeneous polynomials $\cP_{p,q}$ of order $(p,q)$ it holds that
	\begin{align*}
		\cP_{p,q}=\bigoplus_{j=0}^{\min(p,q)}\norm{z}^{2j}\cH_{p-j,q-j},\text{ where }\norm{z}^2=\sum_{j=1}^n{\abs{z_j}^2}=\scalar{z,\zbar}.
	\end{align*}
\end{lemma}
\noindent In order to establish an addition formula for Jacobi polynomials, Koornwinder found in \cite{koorn2} and \cite{koorn3} the reproducing kernel of the space of bihomogeneous harmonics $\cH_{p,q}$. We recall this result for the case of $p\leq q$ in the following theorem.
\begin{theorem}
\label{kernelC}
	For a spherical harmonic $H_{r,s}\in\cH_{r,s}$ it holds that
	\begin{align*}
		\bigscalar{K_{p,q}(z,u),H_{r,s}(z)}_\mS=\delta_{pr}\delta_{qs}H_{r,s}(u)
	\end{align*}	
	with
	\begin{align*}
		K_{p,q}(z,u)=a_{p,q}\scalar{z,\ubar}^{p-\nu}\scalar{\zbar,u}^{q-\nu}\norm{z}^{2\nu}\norm{u}^{2\nu}P_\nu^{n-2,\abs{p-q}}(2s-1),
	\end{align*}
	where $a_{p,q}=\frac{n-1+k}{n-1}\binom{k-\nu+n-2}{k-\nu}$, $k=p+q$, $\nu=\min(p,q)$, the angular variable 
	$s=\frac{|\scalar{z,\ubar}|^2}{\norm{z}^2\norm{u}^2}$ and $P_\nu^{n-2,\abs{p-q}}$ the Jacobi polynomial.
\end{theorem}
As before it is possible to derive the harmonic kernel by a projection of the homogeneous kernel of the same degree. Because any bihomogeneous harmonic of bi-degree $(p,q)$ is also a spherical harmonic of degree $p+q$ we can use the projection operator that is given in $(\ref{oproj})$. Rewriting it in terms of complex variables only changes the corresponding constants.
\begin{theorem}
\label{projkernelC}
The reproducing kernel of $\cH_{p,q}$ can be given as	$K_{p,q}(z,u)=(n)_{p+q}Proj^{p,q}_0\big(Z_{p,q}(z,u)\big)$,
where $Z_{p,q}$ is the homogeneous kernel given in (\ref{zpq}) and the projection from $\cP_{p,q}$ to $\cH_{p-l,q-l}$ is
\begin{align*}
	Proj^{p,q}_{\ell}=\sum_{j=0}^{\min(p,q)-\ell}\beta_j\norm{z}^{2j}\Delta_z^{j+\ell},\text{ with }
	\beta_j=\frac{(-1)^j(n-1+p+q-2\ell)}{j!\ell!}\frac{(n-2+p+q-j-2\ell)!}{(n-1+p+q-\ell)!}.
\end{align*}
\end{theorem}
\section{Symplectic harmonics}
\label{symplectic}
If the dimension of the underlying vector space is doubled once more and is therefore divisible by 4, a quaternionic structure can be applied to identify complex vectors $z=[z_1,\cdots,z_n,z_{n+1},\cdots,z_{2n}]\in\mC^{2n}\simeq\mR^{4n}$ with quaternionic vectors $\tilde{z}=[\tilde{z}_1,\cdots,\tilde{z}_n]\in\mH^n$, with $\tilde{z}_\ell=z_\ell+z_{n+\ell}j$, where $j^2=-1$ and $ij=-ji$. Quaternionic conjugation is denoted by $\tilde{z}^c_\ell=\zbar_\ell-j\zbar_{n+\ell}$ and the quaternionic inner product is defined as
\begin{align*}
	\scalar{\tilde{z},\tilde{u}^c}=\sum_{\ell=1}^n\tilde{z}_\ell\tilde{u}^c_\ell=\scalar{z,\ubar}-\scalar{z,u}_{\hspace{-1.5pt}s}j.
\end{align*}
For the modulus of this inner product one gets
\begin{align*}
\abs{\scalar{\tilde{z},\tilde{u}^c}}^2&=\scalar{\tilde{z},\tilde{u}^c}\big(\scalar{\tilde{z},\tilde{u}^c}\big)^c=\big(\scalar{z,\ubar}-\scalar{z,u}_{\hspace{-1.5pt}s}j\big)\big(\scalar{\zbar,u}+j\overline{\scalar{z,u}}_{\hspace{-1.5pt}s}\big)=\abs{\scalar{z,\ubar}}^2+\abs{\scalar{z,u}_{\hspace{-1.5pt}s}}^2.
\end{align*}
Here $\scalar{z,u}_{\hspace{-1.5pt}s}=\sum_{\ell=1}^nz_\ell u_{n+\ell}-z_{n+\ell}u_\ell$ is an (anti-symmetric) product of the complex vectors $z,u\in\mC^{2n}$. The group under which the quaternionic inner product is invariant is the symplectic group $Sp(n)$ that can be represented as a subgroup of the unitary matrices $U(2n)$. The Laplacian $\Delta_z$, the Euler operators $\mE_z$ and $\bar{\mE}_z$, as well as $r^2=\norm{z}^2$ are invariant under the unitary action and hence under $Sp(n)$. In \cite{symp} the authors derived two additional operators that are invariant under $Sp(n)$. These operators
\begin{align*}
	\cE&=\sum_{j=1}^n\big(z_j\bar{\partial}_{z_{n+j}}-z_{n+j}\diffc{z_j}\big)\\
	\cE^\dagger&=-\sum_{j=1}^n\big(\zbar_j\diff{z_{n+j}}-\zbar_{n+j}\diff{z_j}\big),
\end{align*}
act on the spaces of (complex valued) homogeneous polynomials of bi-degree $(p,q)$ as
\begin{align*}
	\cE:\cP_{p,q}&\rightarrow\cP_{p+1,q-1},\\
	\cE^\dagger:\cP_{p,q}&\rightarrow\cP_{p-1,q+1}.
\end{align*}
With the Euler operators they satisfy the commutator relations of the Lie algebra $\mathfrak{sl}_2$, i.e.
\begin{align}
	\label{one}[\bar{\mE}_z-\mE_z,\cE^\dagger]&=2\cE^\dagger\\
	\label{two}[\bar{\mE}_z-\mE_z,\cE]&=-2\cE\\
	\label{three}[\cE^\dagger,\cE]&=\bar{\mE}_z-\mE_z.
\end{align}
Depending on the degree of homogeneity one can define symplectic homogeneous polynomials as null solutions of exactly one of these operators.
\begin{definition}
\label{operators}
	The space $\cR_{p,q}\subset\cP_{p,q}$ consists of bihomogeneous polynomials $P_{p,q}$ that satisfy
	\begin{align}
		\label{symp1}\cE P_{p,q}&=0,\hspace{1cm}\text{for }p\geq q\\
		\label{symp2}\cE^\dagger P_{p,q}&=0,\hspace{1cm}\text{for }p\leq q.
	\end{align}
	Elements of $\cR_{p,q}$ are called symplectic homogeneous polynomials.
\end{definition}
\noindent Note that for $p<q$ one has that $\cR_{p,q}\cap \ker\cE=\{0\}$ and conversely for $p>q$ that $\cR_{p,q}\cap \ker\cE^\dagger=\{0\}$. The operators $\cE$ and $\cE^\dagger$ and their properties have been studied in \cite{symp}. In the following lemma we restate one of these results.
\begin{lemma}
\label{lem1}
	For the symplectic homogeneous polynomial $R_{p,q}\in\cR_{p,q}$ it holds that
	\begin{align}
		\label{G}\big(\cE\big)^{a}\big(\cE^\dagger)^{b}R_{p,q}
		&=\kappa_{a,b}^{q,p}(\cE^\dagger)^{b-a}R_{p,q},\hspace{10pt}
		\text{for }p\geq q\\
		\label{H}\big(\cE^\dagger\big)^{a}\big(\cE)^{b}R_{p,q}
		&=\kappa_{a,b}^{p,q}(\cE)^{b-a}R_{p,q},\hspace{10pt}\text{for }p\leq q,
	\end{align}
	where
	\begin{align*}
	\kappa_{a,b}^{p,q}=
	\begin{cases}
		0,&a>b\\
		\frac{b!}{(b-a)!}\frac{(q-p-b+a)!}{(q-p-b)!},&a\leq b
	\end{cases}
	\end{align*}
\end{lemma}

\noindent There also exists a Fischer decomposition with respect to the operators $\cE$ and $\cE^\dagger$, see \cite{symp}. There, the authors considered spaces of harmonic symplectic polynomials but the statement and proof remain valid for symplectic polynomials in general. We therefore
give the decomposition of $\cP_{p,q}$ into spaces of symplectic homogeneous polynomials without proof.
\begin{lemma}
	For the space $\cP_{p,q}(\mR^{4n},\mC)$ it holds that
	\begin{align}
		\label{fischer1}\cP_{p,q}&=\bigoplus_{j=0}^q\big(\cE^\dagger)^j\cR_{p+j,q-j}\hspace{10pt}\text{for }p\geq q,\\
		\label{fischer2}\cP_{p,q}&=\bigoplus_{j=0}^p\big(\cE)^j\cR_{p-j,q+j}\hspace{10pt}\text{for }p\leq q.
	\end{align}
\end{lemma}
\noindent In order to avoid redundancy we will assume $p\leq q$ from now on. The results for the case $p\geq q$ can be 
derived by simply exchanging the roles of $p$ and $q$ and those of $\cE$ and $\cE^\dagger$.\\
\indent In order to find a reproducing kernel in this setting we will mimic the construction in the case of complex harmonics given in Theorem \ref{projkernelC}. Instead of projecting the
homogeneous kernel of degree $(p,q)$ onto $\cH_{p,q}$, we will project it onto $\cR_{p,q}$. In the following lemma we construct a differential operator
for such a projection.
\begin{lemma}
\label{projSymp}
	The operator
	\begin{align*}
		Proj^{p,q}_{\cE^\dagger}=\sum_{j=0}^p\gamma_j\big(\cE\big)^j\big(\cE^\dagger\big)^j,
		\hspace{10pt}\text{with }\gamma_j=\frac{(-1)^j}{j!}\frac{(\abs{q-p}+1)!}{(\abs{q-p}+1+j)!}
	\end{align*}
	projects a homogeneous polynomial of order $(p,q)$ (with $p\leq q$) to its symplectic component of the same order, i.e.
	\begin{align*}
		Proj^{p,q}_{\cE^\dagger}\Big(\cE^jR_{p-j,q+j}\Big)=\delta_{0j}\cE^jR_{p-j,q+j}.
	\end{align*}
\end{lemma}
\begin{proof}
	Applying this projector on a homogeneous polynomial results in
	\begin{align*}
		Proj^{p,q}_{\cE^\dagger}\big(P_{p,q}\big)&=\sum_{k=0}^p\gamma_k\big(\cE\big)^k\big(\cE^\dagger\big)^k
		\sum_{j=0}^p\big(\cE)^jR_{p-j,q+j}\\
		&=\sum_{k=0}^p\sum_{j=k}^p\gamma_k\big(\cE\big)^k\big(\cE^\dagger\big)^k\big(\cE)^jR_{p-j,q+j},
	\end{align*}
	with $R_{p-j,q+j}\in\cR_{p-j,q+j}$ according to (\ref{fischer2}). In the second step all terms for $j<k$ vanish due to (\ref{H}). 
	Reordering this double sum gives
	\begin{align*}
		Proj^{p,q}_{\cE^\dagger}\big(P_{p,q}\big)=\sum_{j=0}^p\Big(\sum_{k=0}^j\gamma_k\kappa_{k,j}^{p-j,q+j}\Big)
		\big(\cE\big)^jR_{p-j,q+j}
	\end{align*}
	with $\kappa_{k,j}^{p-j,q+j}=\frac{j!}{(j-k)!}\frac{(q-p+j+k)!}{(q-p+j)!}$ as in Lemma \ref{lem1}. 
	In order for $Proj^{p,q}_{\cE^\dagger}$ to be a projection onto $\cR_{p,q}$ the sum in brackets has to vanish except for $j=0$. 
	We therefore get for the coefficients
	\begin{align*}
		\gamma_0&=1\\
		\sum_{k=0}^{j}\gamma_k\kappa_{k,j}^{p-j,q+j}&=0\hspace{10pt}\text{for }j=1,\ldots,p.
	\end{align*}
	One can easily verify that this linear system of equations is solved by
	\begin{align*}
		\gamma_k=\frac{(-1)^k}{k!}\frac{(\abs{q-p}+1)!}{(\abs{q-p}+1+k)!}.
	\end{align*}
\end{proof}
\begin{remark}
\label{symm}
	For homogeneous polynomials of order $(q,p)$ with still $p\leq q$ it immediately follows that
	\begin{align*}
		Proj^{q,p}_{\cE}=\overline{Proj^{p,q}_{\cE^\dagger}}=\sum_{j=0}^p\gamma_j\big(\cE^\dagger\big)^j\big(\cE\big)^j.
	\end{align*}
\end{remark}
\noindent Our next aim is to compute the projection $Proj^{p,q}_{\cE^\dagger}$ of the bihomogeneous reproducing kernel $Z_{p,q}=\frac{\scalar{z,\ubar}^p\scalar{\zbar,u}^q}{p!q!}$. For reasons of readability we introduce the notations $A=\scalar{z,\ubar}$ and $C=\scalar{z,u}_{\hspace{-1.5pt}s}$. We then get for the actions of $\cE$ and $\cE^\dagger$
\begin{align}
	\label{actionEA}\cE \bar{A}&=\bigg(\sum_{j=1}^nz_j\bar{\partial}_{z_{n+j}}-z_{n+j}
	\bar{\partial}_{z_j}\bigg)\sum_{k=1}^{2n}\zbar_ku_k=
	\sum_{j=1}^nz_ju_{n+j}-z_{n+j}u_j=C\\
	\label{actionEdC}\cE^\dagger C&=-\bigg(\sum_{j=1}^n\zbar_j{\partial}_{z_{n+j}}-\zbar_{n+j}{\partial}_{z_j}\bigg)
	\sum_{k=1}^nz_ku_{n+k}-z_{n+k}u_k=\sum_{j=1}^{2n}\zbar_ju_j=\bar{A},
\end{align}
from which follows that
\begin{align}
	\label{actionEdA}\cE^\dagger A&=-\big(\overline{\cE\bar{A}}\big)=-\bar{C}\\
	\label{actionEC}\cE\bar{C}&=-\big(\overline{\cE^\dagger C}\big)=-A.
\end{align}
\begin{lemma}
\label{projC}
	The action of the projector $Proj^{p,q}_{\cE^\dagger}$ on the kernel $\frac{\scalar{z,\ubar}^p\scalar{\zbar,u}^q}{p!q!}$ for $p\leq q$ is
	\begin{align*}
		Proj^{p,q}_{\cE^\dagger}\bigg(\frac{\scalar{z,\ubar}^p
		\scalar{\zbar,u}^q}{p!q!}\bigg)&=b_{p,q}\scalar{\zbar,u}^{q-p}
		\Big(\abs{\scalar{z,\ubar}}^2+\abs{\scalar{z,u}_{\hspace{-1.5pt}s}}^2\Big)^p\\
		&=b_{p,q}\scalar{\zbar,u}^{q-p}\abs{\scalar{\tilde{z},\tilde{u}^c}}^{2p},
	\end{align*}
	with $b_{p,q}=\frac{q-p+1}{p!(q+1)!}$ and the quaternionic vectors $\tilde{z},\tilde{u}\in\mH^n$ corresponding to the complex vectors $z,u\in\mC^{2n}$.
\end{lemma}
\begin{proof}
	Acting with the projector $Proj^{p,q}_{\cE^\dagger}$ on $Z_{p,q}=\frac{A^p\bar{A}^q}{p!q!}$ gives
	\begin{align*}
		Proj^{p,q}_{\cE^\dagger}\big(Z_{p,q}\big)&=\sum_{j=0}^p\gamma_j\big(\cE\big)^j\big(\cE^\dagger\big)^j
		\frac{A^p\bar{A}^q}{p!q!}\\
		&=\sum_{j=0}^p\gamma_j\big(\cE\big)^j\frac{p!(-1)^j\bar{C}^j}{(p-j)!}
		\frac{A^{p-j}\bar{A}^q}{p!q!},
	\end{align*}
	where equation $(\ref{actionEdA})$ was used $j$ times as well as the fact that $\cE^\dagger \bar{A}=\cE^\dagger \bar{C}=0$. Because also $\cE A=0$ and
	by using equations $(\ref{actionEA})$ and $(\ref{actionEC})$ we further get
	\begin{align*}
		Proj^{p,q}_{\cE^\dagger}\big(Z_{p,q}\big)&=\sum_{j=0}^p\frac{(-1)^j\gamma_j}{(p-j)!q!}A^{p-j}
		\big(\cE\big)^{j-1}\cE\big(\bar{C}^j\bar{A}^q\big)\\
		&=\sum_{j=0}^p\frac{(-1)^j\gamma_j}{(p-j)!q!}A^{p-j}
		\big(\cE\big)^{j-1}\big(qC\bar{A}^{q-1}-jA\bar{C}^{j-1}\big).
	\end{align*}
	The operator $\cE$ acts formally on $\bar{C}^j\bar{A}^q$ as $C\diff{\bar{A}}-A\diff{\bar{C}}$ and $\big(\cE\big)^j$
	thus as $\sum_{k=0}^j{{j}\choose{k}}(-1)^{j-k}C^kA^{j-k}
	\diff{\bar{A}}^k \diff{\bar{C}}^{j-k}$. Hence we get
	\begin{align*}
		Proj^{p,q}_{\cE^\dagger}\big(Z_{p,q}\big)&=\sum_{j=0}^p\frac{(-1)^j\gamma_j}{(p-j)!q!}A^{p-j}
		\sum_{k=0}^j{{j}\choose{k}}(-1)^{j-k}C^kA^{j-k}
		\diff{\bar{A}}^k\diff{\bar{C}}^{j-k}\big(\bar{C}^j\bar{A}^q\big)\\
		&=\sum_{j=0}^p\frac{(-1)^j\gamma_j}{(p-j)!q!}A^{p-j}
		\sum_{k=0}^j{{j}\choose{k}}(-1)^{j-k}C^kA^{j-k}
		\frac{j!q!}{k!(q-k)!}\bar{C}^k\bar{A}^{q-k}.
	\end{align*}
	Using the definition of $\gamma_j$ from Lemma \ref{projSymp}, changing the order of summation and rearranging terms results in
	\begin{align*}
		Proj^{p,q}_{\cE^\dagger}\big(Z_{p,q}\big)&=\frac{A^p\bar{A}^q}{p!q!}\sum_{k=0}^p
		\Big(\frac{C\bar{C}}{A\bar{A}}\Big)^k{q\choose k}\frac{p!(q-p+1)!}{(q+1)!}
		\sum_{j=k}^p(-1)^{j-k}{q+1\choose p-j}{j\choose k}\\
		&=\frac{A^p\bar{A}^q}{p!q!}
		\sum_{k=0}^p\Big(\frac{C\bar{C}}{A\bar{A}}\Big)^k{q\choose k}
		\frac{p!(q-p+1)!}{(q+1)!}{q-k\choose p-k},
	\end{align*}
	where the sum over $j$ is the hypergeometric function $_2F_1(k-p,p-q;1;z)$ evaluated in $z=1$. The desired result then follows as
	\begin{align*}
		Proj^{p,q}_{\cE^\dagger}\big(Z_{p,q}\big)&=\frac{q-p+1}{p!(q+1)!}A^p\bar{A}^q
		\sum_{k=0}^p\Big(\frac{C\bar{C}}{A\bar{A}}\Big)^k{p\choose k}\\
		&=\frac{q-p+1}{p!(q+1)!}A^p\bar{A}^q\Big(1+\frac{C\bar{C}}{A\bar{A}}\Big)^p
		=\frac{q-p+1}{p!(q+1)!}\bar{A}^{q-p}\Big(A\bar{A}+C\bar{C}\Big)^p.
	\end{align*}
\end{proof}
\noindent Given the explicit form of $Z_{p,q}^{S}$, we now show that it acts as a reproducing kernel on the space of symplectic polynomials $\cR_{p,q}$ with
respect to the Fischer inner product.
\begin{theorem}
\label{reproR}
	For a symplectic polynomial $R_{r,s}\in\cR_{r,s}$ it holds that
	\begin{align*}
		\bigscalar{Z_{p,q}^{S}(z,u),\cE^kR_{p-k,q+k}(z)}_\partial=\delta_{k,0}\cE^kR_{p-k,q+k}(u)
	\end{align*}
	with the reproducing kernel
	\begin{align*}
		Z_{p,q}^{S}=b_{p,q}\scalar{\zbar,u}^{q-p}\abs{\scalar{\tilde{z},\tilde{u}^c}}^{2p},
	\end{align*}
	where $b_{p,q}=\frac{q-p+1}{p!(q+1)!}$ and $p\leq q$.
\end{theorem}
\begin{proof}
	We consider the inner product of $Z_{p,q}^{S}$ with a homogeneous polynomial $\cE^kR_{p-k,q+k}\in\cP_{p,q}$
	\begin{align*}
		\bigscalar{Z_{p,q}^{S},\cE^kR_{p-k,q+k}}_\partial&
		=\bigscalar{Proj^{p,q}_{\cE^\dagger}Z_{p,q},\cE^kR_{p-k,q+k}}_\partial,
	\end{align*}
	where we used Lemma \ref{projC}. As the projection operator $Proj^{p,q}_{\cE^\dagger}$ is self-adjoint we get
	\begin{align*}
		\bigscalar{Z_{p,q}^{S},\cE^kR_{p-k,q+k}}_\partial&
		=\bigscalar{Z_{p,q},Proj^{p,q}_{\cE^\dagger}\big(\cE^kR_{p-k,q+k}\big)}_\partial\\
		&=\bigscalar{Z_{p,q},\delta_{k,0}\cE^kR_{p-k,q+k}}_\partial,
	\end{align*}
	where the last equation holds because of the construction of $Proj^{p,q}_{\cE^\dagger}$ in Lemma \ref{projSymp}. As $\cE^kR_{p-k,q+k}$ is a homogeneous
	polynomial, it gets reproduced by the kernel $Z_{p,q}$, hence
	\begin{align*}
		\bigscalar{Z_{p,q}^{S},\cE^kR_{p-k,q+k}}_\partial&=\delta_{k,0}\cE^kR_{p-k,q+k}(u).
	\end{align*}
\end{proof}
\noindent In \cite{symp} the authors considered symplectic polynomials that are in the kernel of the Laplacian, rather than general symplectic polynomials.
\begin{definition}
	The space of symplectic harmonics $\cH_{p,q}^S$ consists of harmonic symplectic polynomials, thus $H_{p,q}\in\cH_{p,q}^S$ if
	\begin{align*}
		\Delta H_{p,q}&=0,\\
		H_{p,q}&\in\cR_{p,q}.
	\end{align*}
\end{definition}
\noindent The main result of this section is to construct a reproducing kernel on $\cH_{p,q}^S$ and describe it in terms of special functions. To do so we employ once more the idea of projecting the homogeneous kernel $Z_{p,q}$, this time onto $\cH_{p,q}^S$. This projection can be done by simply applying
the harmonic operator $Proj^{p,q}_0$ and the symplectic operator $Proj^{p,q}_{\cE^\dagger}$ consecutively. The operators $\cE$ and $\cE^\dagger$ commute
with both $\norm{z}^2$ and $\Delta_z$, and therefore $Proj^{p,q}_0$ and $Proj^{p,q}_{\cE^\dagger}$ also commute, as illustrated in the following diagram.
\begin{center}
\begin{tikzpicture}[thick, main node/.style={}]

\begin{scope}[]

\node[main node, minimum size = 1cm] (A) {$\cP_{p,q}$};
\node[main node, minimum size = 1cm, below=2cm of A] (C) {$\cR_{p,q}$};
\node[main node, minimum size = 1cm, right=3cm of A] (B) {$\cH_{p,q}$};
\node[main node, minimum size = 1cm, right=3cm of C] (D) {$\cH_{p,q}^S$};
\path[every node/.style={font=\sffamily\small}]
    (A) edge[->] node[left] {$Proj^{p,q}_{\cE^\dagger}$} (C)
    (A) edge[->] node[above] {$Proj^{p,q}_{0}$} (B)
	(C) edge[->] node[below] {$Proj^{p,q}_{0}$} (D) 
	(B) edge[->] node[right] {$Proj^{p,q}_{\cE^\dagger}$} (D)
      ;
\end{scope}
\end{tikzpicture}
\end{center}
For the same reason of commutativity both the spaces of spherical harmonics and symplectic polynomials can be decomposed into symplectic harmonics. The first of these two Fischer decompositions (for $p\leq q$)
\begin{align*}
	\cH_{p,q}&=\bigoplus_{j=0}^{p}(\cE)^j\cH_{p-j,q+j}^S\hspace{1cm}\cR_{p,q}=\bigoplus_{j=0}^{p}\norm{z}^{2j}\cH_{p-j,q-j}^S
\end{align*}
is already given in \cite{symp} together with a full decomposition of the space of bihomogeneous polynomials, i.e.
\begin{align*}
	\cP_{p,q}&=\bigoplus_{j=0}^{p}\bigoplus_{k=0}^{p-j}\norm{z}^{2j}(\cE)^k\cH_{p-j-k,q-j+k}^S.
\end{align*}
We will proceed in the same way as for the symplectic kernel by first computing the projection explicitly and showing afterwards that the result is indeed the reproducing kernel. The two possible calculations
\begin{align}
	\label{option1}&Proj^{p,q}_{\cE^\dagger}\Big(Proj^{p,q}_{0}\big(Z_{p,q}\big)\Big)=Proj^{p,q}_{\cE^\dagger}\Big(K_{p,q}\Big)\\
	\label{option2}&Proj^{p,q}_{0}\Big(Proj^{p,q}_{\cE^\dagger}\big(Z_{p,q}\big)\Big)=Proj^{p,q}_{0}\Big(Z_{p,q}^{S}\Big)
\end{align}
yield the same result. Applying the symplectic projection on the harmonic kernel according to (\ref{option1}), however, is a tedious computation. Therefore we proceed by using formula (\ref{option2}). In order to do so, we need the action of the Laplacian on the symplectic kernel $Z_{p,q}^{S}$ which is determined in the following lemma.
\begin{lemma}
\label{actionLap}
	The action of the Laplacian on the kernel $Z_{p,q}^{S}$ is given by
	\begin{align*}
		\Delta_z Z_{p,q}^{S} = \norm{u}^2Z_{p-1,q-1}^{S}.
	\end{align*}
\end{lemma}
\begin{proof}
	The Laplacian commutes with $\cE$, $\cE^\dagger$ and hence with $Proj^{p,q}_{\cE^\dagger}$. Therefore we have
	\begin{align*}
		\Delta Z_{p,q}^{S}&=\Delta\Big(Proj^{p,q}_{\cE^\dagger}Z_{p,q}\Big)\\
		&=Proj^{p,q}_{\cE^\dagger}\Big(\Delta\frac{\scalar{z,\bar{u}}^p\scalar{\zbar,u}^q}{p!q!}\Big)\\
		&=Proj^{p,q}_{\cE^\dagger}\Big(\norm{u}^2\frac{\scalar{z,\bar{u}}^{p-1}\scalar{\zbar,u}^{q-1}}{(p-1)!(q-1)!}\Big).
	\end{align*}
	As the projection $Proj^{p,q}_{\cE^\dagger}$ is independent of $u$ we get
	\begin{align*}
		\Delta Z_{p,q}^{S}&=\norm{u}^2 Proj^{p,q}_{\cE^\dagger}\Big(Z_{p-1,q-1}\Big)\\
		&=\norm{u}^2Z_{p-1,q-1}^{S}.
	\end{align*}
	Note that in the last step we used that $Proj^{p,q}_{\cE^\dagger}\Big(Z_{p-1,q-1}\Big)=Proj^{p-1,q-1}_{\cE^\dagger}\Big(Z_{p-1,q-1}\Big)$ which follows
	from the construction of $Proj^{p,q}_{\cE^\dagger}$ in Lemma \ref{projSymp}.
\end{proof}
\begin{theorem}
\label{projSH}
	Projecting $Z_{p,q}^{S}$ on $\cH_{p,q}^{S}$ results in the polynomial
	\begin{align*}
		Proj^{p,q}_0\Big(Z_{p,q}^{S}\Big)=c_{p,q}\scalar{z,\ubar}^{q-p}\norm{z}^{2p}\norm{u}^{2p}
		P_p^{2n-3,q-p+1}\big(2t-1\big),
	\end{align*}
	with $c_{p,q}=\frac{(q-p+1)(q+2n-2)!}{(p+q+2n-2)!(q+1)!}$, the angular variable 
	$t=\frac{\abs{\scalar{\tilde{z},\tilde{u}^c}}^2}{\norm{z}^2\norm{u}^2}$ and the Jacobi polynomial 
	$P_p^{\alpha,\beta}(x)$.
\end{theorem}
\begin{proof}
	Applying the harmonic projection $Proj^{p,q}_{0}$ on the symplectic kernel $Z_{p,q}^{S}$ results in
	\begin{align*}
		Proj^{p,q}_{0}\Big(Z_{p,q}^{S}\Big)&=\sum_{j=0}^p\beta_j\norm{z}^{2j}\Delta^jZ_{p,q}^{S}\\
		&=\sum_{j=0}^p\beta_j\norm{z}^{2j}\norm{u}^{2j}Z_{p-j,q-j}^{S},
	\end{align*}
	where Lemma \ref{actionLap} was used. With Lemma \ref{projC} and putting 
	$t=\frac{\abs{\scalar{\tilde{z},\tilde{u}^c}}^2}{\norm{z}^2\norm{u}^2}$ we get
	\begin{align*}
		Proj^{p,q}_{0}\Big(Z_{p,q}^{S}\Big)&=\sum_{j=0}^p\beta_j\norm{z}^{2j}\norm{u}^{2j}\frac{q-p+1}{(p-j)!(q-j+1)!}
		\scalar{\zbar,u}^{q-p}\abs{\scalar{\tilde{z},\tilde{u}^c}}^{2(p-j)}\\
		&=\norm{z}^{2p}\norm{u}^{2p}\bar{A}^{q-p}\sum_{j=0}^p\beta_j\frac{q-p+1}{(p-j)!(q-j+1)!}t^{p-j}.
	\end{align*}
	Reversing the summation order and using the definition of $\beta_j$, given in Theorem \ref{projkernelC}, results in
	\begin{align*}
		Proj^{p,q}_{0}\Big(Z_{p,q}^{S}\Big)&=\norm{z}^{2p}\norm{u}^{2p}\bar{A}^{q-p}
		\sum_{j=0}^p\beta_{p-j}\frac{q-p+1}{j!(q-p+j+1)!}t^{j}\\
		&=\norm{z}^{2p}\norm{u}^{2p}\bar{A}^{q-p}\sum_{j=0}^p
		\frac{(-1)^{p-j}(q+2n-2+j)!(q-p+1)}{(p-j)!(p+q+2n-2)!j!(q-p+j+1)!}t^j.
	\end{align*}
	By pulling out the factors that are independent of $j$, the remaining sum can (up to a constant) be recognized as a Jacobi polynomial in $t$, hence
	\begin{align*}
		Proj^{p,q}_{0}\Big(Z_{p,q}^{S}\Big)&=(-1)^p\frac{(q-p+1)}{p!(p+q+2n-2)!}\norm{z}^{2p}\norm{u}^{2p}\bar{A}^{q-p}
		\sum_{j=0}^p\frac{(q+2n-2+j)!}{(q-p+j+1)!}{p\choose j}(-t)^j\\
		&=(-1)^p\frac{(q-p+1)(q+2n-2)!}{(p+q+2n-2)!(q+1)!}\norm{z}^{2p}\norm{u}^{2p}
		\bar{A}^{q-p}P_p^{q-p+1,2n-3}\big(1-2t\big).
	\end{align*}
	The claim then follows by the well-known property of Jacobi polynomials that $(-1)^pP^{\alpha,\beta}_p(x)=P^{\beta,\alpha}_p(-x)$.
\end{proof}
The harmonic projection we computed in the previous lemma is the (unique) reproducing kernel on the space of symplectic harmonics $\cH_{p,q}^S$. We prove this statement in the following theorem.
\begin{theorem}
	For a symplectic harmonic $H_{r,s}^S\in\cH_{r,s}^S$ it holds that
	\begin{align*}
		\bigscalar{K_{p,q}^{S}(z,u),\norm{z}^{2j}\cE^kH_{p-j-k,q-j+k}^S(z)}_\mS=\delta_{j,0}\delta_{k,0}\norm{z}^{2j}\cE^kH_{p-j-k,q-j+k}^S(u)
	\end{align*}
	with the reproducing kernel
	\begin{align*}
		K_{p,q}^{S}=d_{p,q}\scalar{z,\ubar}^{q-p}\norm{z}^{2p}\norm{u}^{2p}P_p^{2n-3,q-p+1}\big(2t-1\big),
	\end{align*}
	where 
	$d_{p,q}=\frac{(q-p+1)(p+q+2n-1)(q+2n-2)!}{(2n-1)!(q+1)!}$, $p\leq q$, $t=\frac{\abs{\scalar{\tilde{z},\tilde{u}^c}}^2}{\norm{z}^2\norm{u}^2}$
	and the Jacobi polynomial $P_p^{\alpha,\beta}(x)$.
\end{theorem}
\begin{proof}
	When considering the spherical inner product of $K_{p,q}^S$ and a homogeneous polynomial $\norm{z}^{2j}\cE^kH_{p-j-k,q-j+k}^S$ we get
	\begin{align*}
		\bigscalar{K_{p,q}^S,\norm{z}^{2j}\cE^kH_{p-j-k,q-j+k}^S}_\mS=\frac{d_{p,q}}{c_{p,q}}\bigscalar{Proj^{p,q}_0Z_{p,q}^S,\norm{z}^{2j}\cE^kH_{p-j-k,q-j+k}^S}_\mS,
	\end{align*}
	according to Theorem \ref{projSH}. The constant $d_{p,q}$ was chosen so that $d_{p,q}=(2n)_{p+q}c_{p,q}$. We therefore get
	\begin{align*}
		\bigscalar{K_{p,q}^S,\norm{z}^{2j}\cE^kH_{p-j-k,q-j+k}^S}_\mS&=
		(2n)_{p+q}\bigscalar{Proj^{p,q}_0Z_{p,q}^S,\norm{z}^{2j}\cE^kH_{p-j-k,q-j+k}^S}_{\mS}\\
		&=\bigscalar{Proj^{p,q}_0Z_{p,q}^S,\norm{z}^{2j}\cE^kH_{p-j-k,q-j+k}^S}_\partial
	\end{align*}
	by the proportionality of the spherical and the Fischer inner product as stated in Lemma \ref{cprop}. As before we can use the self-adjointness of 
	$Proj^{p,q}_0$, hence
	\begin{align*}
		\bigscalar{K_{p,q}^S,\norm{z}^{2j}\cE^kH_{p-j-k,q-j+k}^S}_\mS&=
		\bigscalar{Z_{p,q}^S,Proj^{p,q}_0\big(\norm{z}^{2j}\cE^kH_{p-j-k,q-j+k}^S\big)}_\partial\\
		&=\bigscalar{Z_{p,q}^S,\cE^kProj^{p,q}_0\big(\norm{z}^{2j}H_{p-j-k,q-j+k}^S\big)}_\partial
	\end{align*}
	where we also used the commutativity of $\cE^k$ with both $\Delta_z$ and $Proj^{p,q}_0$. By construction the projection $Proj^{p,q}_0$ acts on 				homogeneous polynomials $\norm{z}^{2j}H_{p-j-k,q-j+k}^S$ as $\delta_{j,0}$ for the spherical harmonic $H_{p-j-k,q-j+k}^S$. We thus have that
		\begin{align*}
		\bigscalar{K_{p,q}^S,\norm{z}^{2j}\cE^kH_{p-j-k,q-j+k}^S}_\mS&=
		\bigscalar{Z_{p,q}^S,\delta_{j,0}\cE^k\norm{z}^{2j}H_{p-j-k,q-j+k}^S}_\partial.
	\end{align*}
	Because $\norm{z}^{2j}H_{p-j-k,q-j+k}^S$ is a symplectic polynomial of order $(p-k,q+k)$, we can use the reproduction property of the symplectic kernel
	$Z_{p,q}^S$ as stated in Theorem \ref{reproR} to get the final result.
\end{proof}
\section{Plane wave formulas}
\label{planewaves}
In this final section we will express the reproducing kernels in the aforementioned settings with respect to plane waves. The latter are of the form 
$p_k(x)=\scalar{x,z}^k$ (with $z\in\mC^m$) and are clearly homogeneous polynomials (in $x$) of homogeneity $k$. These plane waves are harmonic if
\begin{align*}
	\Delta_x(\scalar{x,z}^k)=k(k-1)\scalar{z,z}\scalar{x,z}^{k-2}=0
\end{align*}
and therefore if $z$ is an isotropic vector. More specifically for $z=t+is$ with $t,s\in\mR^m$, we need that
\begin{align*}
	\scalar{z,z}=\scalar{t+is,t+is}=\norm{t}^2-\norm{s}^2+2i\scalar{t,s}=0,
\end{align*}
and hence $\norm{t}=\norm{s}$ and $\scalar{t,s}=0$. The tuple $(t,s)$ thus lies on a Stiefel manifold of order $2$. We consider two types of Stiefel manifolds that we denote as $St^{(1)}$ and $St^{(2)}$. These are isomorphic to the homogeneous spaces
\begin{align*}
	St^{(1)}(\mR^m)&=SO(m)/SO(m-2)\\
	St^{(2)}(\mC^m)&=U(m)/U(m-2)
\end{align*}
for the special orthogonal and the unitary group, respectively. A method to compute integrals over these manifolds is by acting with a certain series of differential operators on the integrand. In \cite{pizzetti} Pizzetti established such a formula for integrals over the sphere. Generalizations of these formulas exist in various frameworks, e.g. the supersphere (see \cite{super}). A Pizzetti formula for Stiefel manifolds was developed in \cite{stiefel}. We state here the result for the first case $St^{(1)}$. 
\begin{theorem}
\label{Pizzetti}
For the polynomial $f(s,t):\mR^m\times\mR^m\rightarrow\mC$ it holds that
\begin{align*}
\frac{1}{\omega_{m-1}\omega_{m-2}}\int_{St^{(1)}}f(s,t)d\sigma(s)d\sigma(t)=\sum_{j=0}^{\infty}\sum_{\ell=0}^{\floor{\frac{j}{2}}}\frac{\Gamma(\frac{m}{2})}{4^j\Gamma(j+\frac{m}{2})}\frac{\Gamma(\frac{m-1}{2})}{\Gamma(\ell+\frac{m-1}{2})}\bigg(\frac{I_1^{j-2\ell}}{(j-2\ell)!}\frac{I_2^\ell}{\ell!}f(s,t)\bigg)\bigg|_{\substack{s=0\\t=0}}
\end{align*}
with $I_1=\Delta_s+\Delta_t$, $I_2=\Delta_s\Delta_t-\scalar{\nabla_s,\nabla_t}^2$  and $\omega_j=\frac{2\pi^{j/2}}{\Gamma(j/2)}$ the surface of the $(j-1)$-dimensional unit sphere $\mS^{j-1}$.
\end{theorem}
\noindent Note that in the above expression $\int_{St^{(1)}}d\sigma(s)d\sigma(t)$ can be interpreted as the integral over the $(m-1)$-dimensional sphere $\mS^{m-1}$ with respect to $t$ and the integral over the $(m-2)$-dimensional subsphere $\mS^{m-2}_\perp$ with respect to $s$, that is perpendicular to $t\in\mS^{m-1}$.\\
Our aim is to integrate the plane wave $\scalar{x,z}^k\scalar{y,\zbar}^k$ over the manifold $St^{(1)}$ to show that the result equals the reproducing kernel $K_k(x,y)$ of the spherical harmonics. In order to apply the previous theorem it is necessary to know the actions of the operators $I_1$ and $I_2$ on these plane waves. To simplify computations we will let them act on $f_k(z)=\frac{\scalar{x,z}^k}{k!}\frac{\scalar{y,\zbar}^k}{k!}$.
\begin{lemma}
The operators $I_1=\Delta_s+\Delta_t$ and $I_2=\Delta_s\Delta_t-\scalar{\nabla_s,\nabla_t}^2$ act on the function $f_k(z)=\frac{\scalar{x,z}^k}{k!}\frac{\scalar{y,\zbar}^k}{k!}$ with $z=t+is$ as
\begin{align*}
	I_1(f_k)&=4\scalar{x,y}f_{k-1}\\
	I_2(f_k)&=4\big(\scalar{x,y}^2-\norm{x}^2\norm{y}^2\big)f_{k-2}.
\end{align*}
\end{lemma}
\begin{proof}
	To simplify the computations we express the operators $I_1$ and $I_2$ with respect to the complex gradients $\nabla_z=\frac{1}{2}(\nabla_t-i\nabla_s)$
	and $\nabla_{\zbar}=\frac{1}{2}(\nabla_t+i\nabla_s)$. Therefore, with $\nabla_t=\nabla_z+\nabla_{\zbar}$ and 
	$\nabla_s=i(\nabla_z-\nabla_{\zbar})$, we have
	\begin{align*}
		I_1&=\Delta_t+\Delta_s=4\Delta_z\\
		I_2&=\Delta_t\Delta_s-\scalar{\nabla_t,\nabla_s}^2=-4\bigg(\Delta_z^2+\scalar{\nabla_z,\nabla_z}\scalar{\nabla_{\zbar},\nabla_{\zbar}}\bigg).
	\end{align*}
	The action of $I_1$ on $f_k(z)=\frac{\scalar{x,z}^k}{k!}\frac{\scalar{y,\zbar}^k}{k!}$ is now easily found to be
	\begin{align*}
	I_1\big(f_k\big)&=4\Delta_z\Big(\frac{\scalar{x,z}^k}{k!}\frac{\scalar{y,\zbar}^k}{k!}\Big)=4\sum_{j=1}^m\diff{z_j}\diff{\zbar_j}							\Big(\frac{\scalar{x,z}^k}{k!}\frac{\scalar{y,\zbar}^k}{k!}\Big)\\
	&=4\sum_{j=1}^mx_jy_j\frac{\scalar{x,z}^{k-1}}{(k-1)!}\frac{\scalar{y,\zbar}^{k-1}}{(k-1)!}=4\scalar{x,y} f_{k-1}.
	\end{align*}
	For the action of $I_2$ on $f_k(z)$ it holds that
	\begin{align*}
		I_2\big(f_k\big)&=4\big(\Delta_z^2-\scalar{\nabla_z,\nabla_z}\scalar{\nabla_{\zbar},\nabla_{\zbar}}\big)
		\Big(\frac{\scalar{x,z}^k}{k!}\frac{\scalar{y,\zbar}^k}{k!}\Big)\\
		&=\frac{1}{4}I_1^2f_k(z)-4\sum_{\ell=1}^m\diff{z_\ell}^2\sum_{j=1}^m\diff{\zbar_j}^2\frac{\scalar{x,z}^k}{k!}\frac{\scalar{y,\zbar}^k}{k!}.
	\end{align*}
	In both terms the degree of $f_k$ is lowered by $2$, hence
	\begin{align*}
	I_2(f_k)	&=4\scalar{x,y}^2f_{k-2}(z)-4\norm{x}^2\norm{y}^2f_{k-2}\\
	&=4\Big(\scalar{x,y}^2-\norm{x}^2\norm{y}^2\Big)f_{k-2}.
	\end{align*}
\end{proof}
\noindent Now we can show that the integral of $\scalar{x,t+is}^k\scalar{y,t-is}^k$ over $St^{(1)}$ is indeed the reproducing kernel $K_k(x,y)$. Note that this result was already established in \cite{radon} in a different way in the framework of an inverse Szeg\H{o}-Radon projection. We will prove the statement in the following theorem through a direct computation by means of the generalized Pizzetti formula in Theorem \ref{Pizzetti}.
\begin{theorem}
\label{Frank}
The reproducing kernel of the space of spherical harmonics of degree $k$ can be expressed as
\begin{align}
	\label{kernel}
	K_k(x,y)=\frac{\lambda_k\dim{\cH_k}}{\omega_{m-1}\omega_{m-2}}\int_{St^{(1)}}\scalar{x,t+is}^k\scalar{y,t-is}^kd\sigma(s)d\sigma(t),
\end{align}
with  $\lambda_k={k+\mu\choose\mu}$, $\mu=\frac{m}{2}-1$ and $\dim{\cH_k}$ as in (\ref{dim}).
\end{theorem}
\begin{proof}
The operators $I_1$ and $I_2$ lower the degree of homogeneity (in $z$ and $\zbar$) by $1$ and $2$ respectively. Acting with $I_1^{j-2\ell}I_2^\ell$ on $f_k(z)=\frac{\scalar{x,z}^k}{k!}\frac{\scalar{y,\zbar}^k}{k!}$ according to Theorem \ref{Pizzetti} yields a non-zero result if and only if $(j-2\ell)+2\ell=k$ and therefore $j=k$. Thus with $\vartheta_\ell=\frac{\Gamma(\frac{m}{2})}{\Gamma(k+\frac{m}{2})}\frac{\Gamma(\frac{m-1}{2})}{\Gamma(\ell+\frac{m-1}{2})}$ and
\begin{align}
\label{Lk}
	L_k(x,y) = \frac{1}{\omega_{m-1}\omega_{m-2}}\int_{St^{(1)}}f_k(z)d\sigma(t)d\sigma(s)
\end{align} 
Theorem \ref{Pizzetti} reads as
\begin{align*}
	L_k(x,y)=\sum_{\ell=0}^{\floor{\frac{k}{2}}}\frac{\vartheta_\ell}{4^k}\bigg(\frac{I_1^{k-2\ell}}{(k-2\ell)!}
	\frac{I_2^\ell}{\ell!}f_k(z)\bigg)\bigg|_{z=0}.
\end{align*}
Applying $I_1$ and $I_2$ repeatedly according to the previous lemma results in
\begin{align*}
L_k(x,y)&=\sum_{\ell=0}^{\floor{\frac{k}{2}}}
\frac{\vartheta_\ell}{4^k(k-2\ell)!\ell!}4^\ell\Big(\scalar{x,y}^2-\norm{x}^2\norm{y}^2\Big)^\ell\bigg(4^{k-2\ell}\scalar{x,y}^{k-2\ell}\bigg)\\
&=\sum_{\ell=0}^{\floor{\frac{k}{2}}}\frac{\vartheta_\ell}{4^\ell(k-2\ell)!\ell!}\Big(\scalar{x,y}^2-\norm{x}^2\norm{y}^2\Big)^\ell\scalar{x,y}^{k-2\ell}.
\end{align*}
The powers of $\norm{x}$ and $\norm{y}$ can be pulled out of the summation, hence
\begin{align*}
	L_k(x,y)&=\norm{x}^k\norm{y}^k\sum_{\ell=0}^{\floor{\frac{k}{2}}}
	\frac{\vartheta_\ell}{4^\ell(k-2\ell)!\ell!}\Big(\scalar{\xi,\eta}^2-1\Big)^\ell\scalar{\xi,\eta}^{k-2\ell},
\end{align*}
with $\xi=\frac{x}{\norm{x}}$ and $\eta=\frac{y}{\norm{y}}$. Using the binomial formula to compute $\Big(\scalar{\xi,\eta}^2-1\Big)^\ell$ gives
\begin{align*}
	L_k(x,y)&=\norm{x}^k\norm{y}^k\sum_{\ell=0}^{\floor{\frac{k}{2}}}
	\frac{\vartheta_\ell\scalar{\xi,\eta}^{k-2\ell}}{4^\ell(k-2\ell)!\ell!}\sum_{\rho=0}^\ell{\ell\choose\rho}(-1)^{\rho}\scalar{x,y}^{2(\ell-\rho)}\\
	&=\norm{x}^k\norm{y}^k\sum_{\ell=0}^{\floor{\frac{k}{2}}}\sum_{\rho=0}^\ell\frac{\vartheta_\ell{\ell\choose\rho}(-1)^{\rho}\scalar{\xi,\eta}^{k-2\rho}}
	{4^\ell(k-2\ell)!\ell!}
\end{align*}
Changing the order of summation and using the definition of $\vartheta_\ell$ results in
\begin{align*}
L_k(x,y)&=\norm{x}^k\norm{y}^k\frac{\Gamma(\frac{m}{2})\Gamma(\frac{m-1}{2})}{\Gamma(k+\frac{m}{2})}\sum_{\rho=0}^{\floor{\frac{k}{2}}}\frac{(-1)^{\rho}(2\scalar{\xi,\eta})^{k-2\rho}}{2^{k-2\rho}}\sum_{\ell=\rho}^{\floor{\frac{k}{2}}}\frac{{\ell\choose\rho}}{4^\ell(k-2\ell)!\ell!\Gamma(\ell+\frac{m-1}{2})}.
\end{align*}
One can easily show that the sum over $\ell$ can be expressed in terms of a hypergeometric function $_2F_1$ evaluated in $1$ as
\begin{align*}
	\sum_{\ell=\rho}^{\floor{\frac{k}{2}}}\frac{{\ell\choose\rho}}{4^\ell(k-2\ell)!\ell!\Gamma(\ell+\frac{m-1}{2})}
	=\frac{{}_2F_1(\rho-\frac{k}{2},\rho-\frac{k-1}{2};\rho+\frac{m-1}{2};1)}
	{4^\rho\rho!\Gamma\big(\rho+\frac{m-1}{2}\big)(k-2\rho)!}.
\end{align*}
Using Gauss's hypergeometric theorem we find that
\begin{align*}
	L_k(x,y)&=\norm{x}^k\norm{y}^k
	\frac{\Gamma(\frac{m-1}{2})\Gamma(\frac{m}{2})}{\Gamma(k+\frac{m}{2})}\sum_{\rho=0}^{\floor{\frac{k}{2}}}
	\frac{(-1)^{\rho}(2\scalar{\xi,\eta})^{k-2\rho}}{2^{k-2\rho}}\Bigg(\frac{1}{8\sqrt{\pi}}
	\frac{\Gamma(\frac{m}{2}-1+k-\rho)2^{m+k-2\rho}}{(k-2\rho)!\rho!(m+k-3)!}\Bigg)\\
	&=\norm{x}^k\norm{y}^k\frac{\Gamma(\frac{m-1}{2})\Gamma(\frac{m}{2})\Gamma(\frac{m}{2}-1)}{\Gamma(k+\frac{m}{2})(m+k-3)!}\frac{2^{m-3}}{\sqrt{\pi}}			\sum_{\rho=0}^{\floor{\frac{k}{2}}}(-1)^{\rho}\frac{\Gamma(\frac{m}{2}-1+k-\rho)}{\Gamma(\frac{m}{2}-1)(k-2\rho)!\rho!}(2\scalar{\xi,\eta})^{k-2\rho}.
\end{align*}
We divided and multiplied by $\Gamma(\frac{m}{2}-1)$ to find the sum over $\rho$ to be the Gegenbauer polynomial of degree $k$ with parameter 
$\mu=\frac{m}{2}-1$. This concludes the proof as
\begin{align*}
L_k(x,y)&=\frac{1}{\lambda_k\dim{\cH_k}(k!)^2}\frac{k+\mu}{\mu}\norm{x}^k\norm{y}^kC_k^\mu\big(\scalar{\xi,\eta}\big)=\frac{1}{\lambda_k\dim{\cH_k}(k!)^2}K_k(x,y).
\end{align*}
\end{proof}
\begin{remark}
	Note that the proportionality of $K_k(x,y)$ and the integral of $\scalar{x,z}^k\scalar{y,\zbar}^k$ over $St^{(1)}$ follow immediately from 
	the uniqueness of the reproducing kernel as the zonal harmonic according to \cite{stein}. With $L_k(x,y)$ as defined in (\ref{Lk}) it is clear that 
	$L_k\in\cH_k$. In order to be zonal, i.e. to only depend on $\scalar{x,y}$, it is sufficient to show that $L_k(x,y)=L_k(Rx,Ry)$ for any
	$R\in SO(m)$. This follows directly from the symmetry of $St^{(1)}$ and we have thus that $K_k(x,y)$ is proportional to $L_k(x,y)$. 
\end{remark}
An alternative representation of the kernel $K_k(x,y)$ has been found by Sherman in \cite{sherman} (see also \cite{yang}). Here a rational function is integrated over a certain $(m-2)$-dimensional sphere. We state this result in the following theorem, so the reader may compare it with Theorem \ref{Frank}.
\begin{theorem}
	For $t\in\mS^{m-1}$, the sphere of dimension $m-2$ that is perpendicular to $t$ is denoted by $\mS^{m-2}_\perp$. When $\xi\in\mS^{m-1}$ and 
	$\eta\in\mS^{m-1}\setminus\mS^{m-2}_\perp$ it holds that
	\begin{align*}
		K_k(\xi,\eta)=\frac{\dim{\cH_k}}{\omega_{m-2}}\int_{\mS^{m-2}_\perp}\frac{(sgn{\scalar{\eta,s}})^{m-2}
		\scalar{\xi,t+is}^k}{\scalar{\eta,t+is}^{k+m-2}}d\sigma(s).
	\end{align*}
\end{theorem}
When considering complex plane waves of the form $\scalar{z,\overline{s+t}}^p\scalar{\zbar,s-t}^q$, with $z,s,t\in\mC^n$, one can use the same strategy as before to find an integral representation of the reproducing kernel $K_{p,q}(z,u)$ of complex harmonics. Complex harmonic plane waves have to satisfy
\begin{align*}
	\Delta_z\big(\scalar{z,\overline{s+t}}^p\scalar{\zbar,s-t}^q\big)
	&=pq\scalar{s-t,\overline{s+t}}\scalar{z,\overline{s+t}}^{p-1}\scalar{\zbar,s-t}^{q-1}=0,
\end{align*}
and hence
\begin{align*}
	\scalar{s-t,\overline{s+t}}=\norm{s}^2-\norm{t}^2+\scalar{s,\bar{t}}-\scalar{\bar{s},t}=0.
\end{align*}
This is the case if the tuple $(t,s)$ lies on the complex Stiefel manifold $St^{(2)}$. Here the complex vector $t$ is on the unit sphere $\mS^{2n-1}$ of (real) dimension $2n-1$ and $s$ on the sphere $\mS^{2n-3}$ that is orthogonal to $t$ with respect to the Hermitian inner product. A Pizzetti formula for integrals over $St^{(2)}$, similar to Theorem \ref{Pizzetti}, can be found in \cite{stiefel}. We restate this result for the case of polynomials
in the following theorem.
\begin{theorem}
\label{Pizzetti2}
For a polynomial $f(v,w):\mR^{2n}\times \mR^{2n}\rightarrow\mC$ it holds that
\begin{align*}
\frac{1}{\omega_{2n-1}\omega_{2n-3}}\int_{St^{(2)}}f(v,w)d\sigma(v)d\sigma(w)=
\sum_{j=0}^\infty\sum_{\ell=0}^{\floor{\frac{j}{2}}}\frac{(n-1)!}{4^j(n+j-1)!}\frac{(n-2)!}{(n+\ell-2)!}
\bigg(\frac{I_1^{j-2\ell}}{(j-2\ell)!}\frac{I_2^\ell}{\ell!}f(v,w)\bigg)\bigg|_{\substack{w=0\\v=0}}
\end{align*}
with $v,w\in\mR^{2n}$, $I_1=\Delta_w+\Delta_v$, $I_2=\Delta_w\Delta_v-\scalar{\nabla_w,\nabla_v}^2-\scalar{\nabla_w,J\nabla_v}^2$ and $J=\begin{bmatrix}0 & \mathbb{I}_n\\-\mathbb{I}_n &0\end{bmatrix}$.
\end{theorem}
\noindent Our aim is to express the reproducing kernel $K_{p,q}$ as an integral of plane waves over $St^{(2)}$. In order to apply the previous theorem for this computation it is useful to know the action of the operators $I_1$ and $I_2$ on complex plane waves.
\begin{lemma}
\label{action}
	The action of the operators $I_1$ and $I_2$ as stated in Theorem \ref{Pizzetti2} on the function 
	\begin{align*}
		f_{p,q}(s,t)=\frac{\scalar{z,\overline{t+s}}^p}{p!}\frac{\scalar{\zbar,t-s}^q}{q!}
		\frac{\scalar{u,\overline{t-s}}^q}{q!}\frac{\scalar{\ubar,t+s}^p}{p!}
	\end{align*}
	is given by
	\begin{align*}
		I_1(f_{p,q})&=8\big(\scalar{z,\ubar}f_{p-1,q}+\scalar{\zbar,u}f_{p,q-1}\big)\\
		I_2(f_{p,q})&=64\big(\abs{\scalar{z,\ubar}}^2-\norm{z}^2\norm{u}^2\big)f_{p-1,q-1}.
	\end{align*}
\end{lemma}
\begin{proof}
	The operators $I_1$ and $I_2$ are given with respect to $w=\col{w_1}{w_2}$ and $v=\col{v_1}{v_2}$, where $w_j, v_j\in\mR^n$ for $j=1,2$. 
	With the complex gradients $\nabla_s=\frac{1}{2}(\nabla_{w_1}-i\nabla_{w_2})$ and $\nabla_t=\frac{1}{2}(\nabla_{v_1}-i\nabla_{v_2})$ we get for
	the operators $I_1$ and $I_2$ that
	\begin{align*}
		I_1&=\Delta_w+\Delta_v=4(\scalar{\nabla_s,\nabla_{\sbar}}+\scalar{\nabla_t,\nabla_{\tbar}})=4(\Delta_s+\Delta_t)\\
		I_2&=\Delta_w\Delta_v-\scalar{\nabla_w,\nabla_v}^2-\scalar{\nabla_w,J\nabla_v}^2
		=16\Big(\Delta_s\Delta_t-\scalar{\nabla_s,\nabla_{\tbar}}\scalar{\nabla_{\sbar},\nabla_t}\Big).
	\end{align*}
	When denoting with $f_{p,q}(s,t)=A_pB_qC_qD_p$ where
	\begin{align*}
		A_j&=\frac{\scalar{z,\overline{t+s}}^j}{j!}\hspace{1cm}B_j=\frac{\scalar{\zbar,t-s}^j}{j!}\hspace{1cm}
		C_j=\frac{\scalar{u,\overline{t-s}}^j}{j!}\hspace{1cm}D_j=\frac{\scalar{\ubar,t+s}^j}{j!}
	\end{align*}
	we get for the Laplacians with respect to $s$ and $t$
	\begin{align*}
		\Delta_s\big(f_{p,q}\big)&=\sum_{j=1}^n\diff{s_j}\diff{\sbar_j}\big(A_pB_qC_qD_p\big)
		=\sum_{j=1}^n\diff{s_j}\big(z_jA_{p-1}B_qC_qD_p-u_jA_pB_qC_{q-1}D_p\big)\\
		&=\sum_{j=1}^n-z_j\zbar_jA_{p-1}B_{q-1}C_qD_p+z_j\ubar_jA_{p-1}B_qC_qD_{p-1}
		-u_j\ubar_jA_pB_qC_{q-1}D_{p-1}+\zbar_ju_jA_pB_{q-1}C_{q-1}D_p\\
		&=-\norm{z}^2A_{p-1}B_{q-1}C_qD_p+\scalar{z,\ubar}A_{p-1}B_qC_qD_{p-1}-\norm{u}^2A_pB_qC_{q-1}D_{p-1}+\scalar{\zbar,u}A_pB_{q-1}C_{q-1}D_p
	\end{align*}
	and equivalently
	\begin{align*}
		\Delta_t\big(f_{p,q}\big)=\norm{z}^2A_{p-1}B_{q-1}C_qD_p+\scalar{z,\ubar}A_{p-1}B_qC_qD_{p-1}
		+\norm{u}^2A_pB_qC_{q-1}D_{p-1}+\scalar{\zbar,u}A_pB_{q-1}C_{q-1}D_p.
	\end{align*}
	The action of $I_1$ on $f_{p,q}$ is thus
	\begin{align*}
		I_1\big(f_{p,q}\big)=4\big(\Delta_s+\Delta_t\big)\big(f_{p,q}\big)=8\Big(\scalar{z,\ubar}f_{p-1,q}+\scalar{\zbar,u}f_{p,q-1}\Big).
	\end{align*}
	The action of $I_2$ on $f_{p,q}$ is computed in a similar way. As this calculation is quite tedious while providing no further insight,
	we omit it.
\end{proof}
\begin{theorem}
\label{planeC}
	The reproducing kernel of the space of complex harmonics of degree $(p,q)$ can be expressed as
	\begin{align*}
		K_{p,q}(z,u)=\frac{\lambda_{p,q}\dim{\cH_{p,q}}}{\omega_{2n-1}\omega_{2n-3}}\int_{St^{(2)}}
		\scalar{z,\overline{t+s}}^p\scalar{\zbar,t-s}^q\scalar{u,\overline{t-s}}^q\scalar{\ubar,t+s}^pd\sigma(s)d\sigma(t),
	\end{align*}
	where $\lambda_{p,q}=\frac{(k+n-1)!}{2^k(n-1)!(k-\nu)!}$, $k=p+q$, $\nu=\min(p,q)$, $\dim{\cH_{p,q}}$ 
	as in (\ref{dimC}) and $s,t,z,u\in\mC^{n}$.
\end{theorem}
\begin{proof}
	With $f_{p,q}(s,t)=A_pB_qC_qD_p$ as in the previous lemma, Pizzetti's formula gives
	\begin{align*}
		&\frac{1}{\omega_{2n-1}\omega_{2n-3}}\int_{St^{(2)}}f_{p,q}(s,t)d\sigma(s)d\sigma(t)
		=\sum_{j=0}^{\infty}\sum_{\ell=0}^{\floor{\frac{j}{2}}}\frac{(n-1)!}{4^j(n+j-1)!}\frac{(n-2)!}{(n+\ell-2)!}
		\bigg(\frac{I_1^{j-2\ell}}{(j-2\ell)!}\frac{I_2^\ell}{\ell!}f_{p,q}\bigg)\bigg|_{\substack{s=0\\t=0}}.
	\end{align*}
	Note that $I_j\big(f_{p,0}\big)=I_j\big(f_{0,q}\big)=0$ for $j=1,2$. In Lemma \ref{action} we see that the operators $I_1$ and $I_2$ can act at 
	most $p+q$ and $\nu=\min(p,q)$ times on $f_{p,q}$ respectively. When denoting with $\nu_j=\min(\floor{\frac{j}{2}},\nu)$ we get
	\begin{align*}
		&\frac{1}{\omega_{2n-1}\omega_{2n-3}}\int_{St^{(2)}}f_{p,q}(s,t)d\sigma(s)d\sigma(t)\\
		&=\sum_{j=0}^{p+q}\sum_{\ell=0}^{\nu_j}\frac{(n-1)!}{4^j(n+j-1)!}\frac{(n-2)!}{(n+\ell-2)!}
		\frac{64^{\ell}}{(j-2\ell)!\ell!}
		\Big(\scalar{z,\ubar}\scalar{\zbar,u}-\norm{z}^2\norm{u}^2\Big)^\ell\bigg(I_1^{j-2\ell}f_{p-\ell,q-\ell}\bigg)\bigg|_{\substack{s=0\\t=0}}\\
		&=\sum_{j=0}^{p+q}\sum_{\ell=0}^{\nu_j}\frac{(n-1)!}{(n+j-1)!}\frac{(n-2)!}{(n+\ell-2)!}
		\frac{2^{j}}{(j-2\ell)!\ell!}\Big(\scalar{z,\ubar}\scalar{\zbar,u}-\norm{z}^2\norm{u}^2\Big)^\ell\\
		&\times\sum_{a=0}^{j}{{j-2\ell}\choose{a}}\scalar{z,\ubar}^a\scalar{\zbar,u}^{j-2\ell-a}
		\bigg(f_{p-\ell-a,q+\ell+a-j}\bigg)\bigg|_{\substack{s=0\\t=0}}.
	\end{align*}
	The expression $\big(f_{p-\ell-a,q+\ell+a-j}\big)\big|_{\substack{s=0\\t=0}}$ is non-zero if and only if 
	$a=p-\ell$ and $j=p+q$. We therefore have that $\nu=\min(p,q)\leq\floor{\frac{p+q}{2}}$ and hence
	\begin{align*}
		&\frac{1}{\omega_{2n-1}\omega_{2n-3}}\int_{St^{(2)}}f_{p,q}(s,t)d\sigma(s)d\sigma(t)\\
		&=\sum_{\ell=0}^{\nu}\frac{(n-1)!}{(p+q+n-1)!}\frac{(n-2)!}{(n+\ell-2)!}
		\frac{2^{p+q}}{(p+q-2\ell)!\ell!}{{p+q-2\ell}\choose{p-l}}\Big(\scalar{z,\ubar}\scalar{\zbar,u}-\norm{z}^2\norm{u}^2\Big)^\ell
		\scalar{z,\ubar}^{p-\ell}\scalar{\zbar,u}^{q-\ell}\\
		&=\sum_{\ell=0}^{\nu}\frac{(n-1)!}{(p+q+n-1)!}\frac{(n-2)!}{(n+\ell-2)!}
		\frac{2^{p+q}}{\ell!(q-l)!(p-l)!}\Big(\scalar{z,\ubar}\scalar{\zbar,u}-\norm{z}^2\norm{u}^2\Big)^\ell
		\scalar{z,\ubar}^{p-\ell}\scalar{\zbar,u}^{q-\ell}.
	\end{align*}
	Because $\nu=\min(p,q)$ and with $k=p+q$ we have that $(q-l)!(p-l)!=(\nu-l)!(k-\nu-l)!$. When pulling powers of $\scalar{z,\ubar}$ and
	$\scalar{\zbar,u}$ out of the sum, and with $\xi=\frac{z}{\norm{z}}$ and $\eta=\frac{u}{\norm{u}}$ it holds that
	\begin{align*}
		&\frac{1}{\omega_{2n-1}\omega_{2n-3}}\int_{St^{(2)}}f_{p,q}(s,t)d\sigma(s)d\sigma(t)\\
		&=\scalar{z,\ubar}^{p-\nu}\scalar{\zbar,u}^{q-\nu}\frac{2^{k}(n-1)!(n-2)!}{(k+n-1)!(\nu+n-2)!\nu!(k-\nu)!}\sum_{\ell=0}^{\nu}
		{\nu+n-2\choose{\nu-l}}{k-\nu\choose{\ell}}\Big(\scalar{z,\ubar}\scalar{\zbar,u}-\norm{z}^2\norm{u}^2\Big)^\ell
		\big(\scalar{z,\ubar}\scalar{\zbar,u}\big)^{\nu-\ell}\\
		&=\scalar{z,\ubar}^{p-\nu}\scalar{\zbar,u}^{q-\nu}\norm{z}^{2\nu}\norm{u}^{2\nu}\frac{2^{k}(n-1)!(n-2)!}{(k+n-1)!(\nu+n-2)!p!q!}\sum_{\ell=0}^{\nu}
		{\nu+n-2\choose{\nu-l}}{k-\nu\choose{\ell}}\bigg(\frac{\scalar{z,\ubar}\scalar{\zbar,u}-\norm{z}^2\norm{u}^2}{\scalar{z,\ubar}
		\scalar{\zbar,u}}\bigg)^\ell\abs{\scalar{\xi,\bar{\eta}}}^{2\nu}.
	\end{align*}
	With $a_{p,q}=\frac{n-1+k}{n-1}{k-\nu+n-2\choose{k-\nu}}$ (as in Theorem \ref{kernelC}) we have
	\begin{align*}
		&\frac{1}{\omega_{2n-1}\omega_{2n-3}}\int_{St^{(2)}}f_{p,q}(s,t)d\sigma(s)d\sigma(t)\\
		&=\scalar{z,\ubar}^{p-\nu}\scalar{\zbar,u}^{q-\nu}\norm{z}^{2\nu}\norm{u}^{2\nu}\frac{a_{p,q}}{\lambda_{p,q}\dim{\cH_{p,q}}(p!q!)^2}\sum_{\ell=0}^{\nu}
		{\nu+n-2\choose{\nu-l}}{k-\nu\choose{\ell}}\bigg(1-\frac{1}{\abs{\scalar{\xi,\bar{\eta}}}^2}\bigg)^\ell
		\abs{\scalar{\xi,\bar{\eta}}}^{2\nu}.
	\end{align*}
	By denoting $\Phi=2\abs{\scalar{\xi,\bar{\eta}}}^2-1$ and thus $\abs{\scalar{\xi,\bar{\eta}}}^2=\frac{\Phi+1}{2}$ we have
	\begin{align*}
		&\frac{1}{\omega_{2n-1}\omega_{2n-3}}\int_{St^{(2)}}f_{p,q}(s,t)d\sigma(s)d\sigma(t)\\
		&=\scalar{z,\ubar}^{p-\nu}\scalar{\zbar,u}^{q-\nu}\norm{z}^{2\nu}\norm{u}^{2\nu}\frac{a_{p,q}}{\lambda_{p,q}\dim{\cH_{p,q}}(p!q!)^2}\sum_{\ell=0}^{\nu}
		{\nu+n-2\choose{\nu-\ell}}{k-\nu\choose{\ell}}\bigg(\frac{\Phi-1}{\Phi+1}\bigg)^\ell\bigg(\frac{\Phi+1}{2}\bigg)^\nu.
	\end{align*}
	When reversing the direction of the summation
	\begin{align*}
		&\frac{1}{\omega_{2n-1}\omega_{2n-3}}\int_{St^{(2)}}f_{p,q}(s,t)d\sigma(s)d\sigma(t)\\
		&=\scalar{z,\ubar}^{p-\nu}\scalar{\zbar,u}^{q-\nu}\norm{z}^{2\nu}\norm{u}^{2\nu}\frac{a_{p,q}}{\lambda_{p,q}\dim{\cH_{p,q}}(p!q!)^2}\sum_{\ell=0}^{\nu}
		{\nu+n-2\choose{\ell}}{\nu+\abs{p-q}\choose{\nu-\ell}}\bigg(\frac{\Phi-1}{2}\bigg)^{\nu-\ell}\bigg(\frac{\Phi+1}{2}\bigg)^{\ell}\\
		&=\scalar{z,\ubar}^{p-\nu}\scalar{\zbar,u}^{q-\nu}\norm{z}^{2\nu}\norm{u}^{2\nu}\frac{a_{p,q}}{\lambda_{p,q}\dim{\cH_{p,q}}(p!q!)^2}P_{\nu}^{n-2,|p-q|}(\Phi)
	\end{align*}
	we find the Jacobi polynomial $P_{\nu}^{n-2,|p-q|}(\Phi)$ of degree $\nu=\min(p,q)$, evaluated in $\Phi=2\abs{\scalar{\xi,\bar{\eta}}}^2-1$.
\end{proof}
From Theorem \ref{planeC} one can derive a plane wave representation of the reproducing kernel of symplectic harmonics $K_{p,q}^S$ that we constructed in Section \ref{symplectic}. This can be done by applying the projection operator $Proj^{p,q}_{\cE^\dagger}$ that we described in Lemma \ref{projSymp}.
\begin{theorem}
	The reproducing kernel of the space of symplectic harmonics can be represented as
	\begin{align*}
		K_{p,q}^S(z,u)=\frac{\lambda_{p,q}}{\omega_{4n-1}\omega_{4n-3}}\int_{St^{(2)}}g_z^{p,q}(s,t)\overline{g_u^{p,q}(s,t)}d\sigma(s,t),
	\end{align*}
	with $\lambda_{p,q}$ as in Theorem \ref{planeC} (for complex dimension $2n$) and the symplectic plane wave\\
	\begin{align*}
		g_z^{p,q}(s,t)=\frac{q-p+1}{p!}{(q+1)!}\scalar{\zbar,t+s}^{q-p}\Big(\scalar{z,\bar{t}-\bar{s}}
		\scalar{\zbar,t+s}+\scalar{z,t+s}_{\hspace{-1.5pt}s}\scalar{\zbar,\bar{t}-\bar{s}}_{\hspace{-1.5pt}s}\Big)^p.
	\end{align*}
\end{theorem}
\begin{proof}
	We can express the complex reproducing kernel $K_{p,q}$ (for dimension $4n$) in terms of plane waves according to Theorem \ref{planeC} as
	\begin{align*}
		K_{p,q}(z,u)=\frac{\lambda_{p,q}}{\omega_{4n-1}\omega_{4n-3}}\int_{St^{(2)}}
		\scalar{z,\overline{t+s}}^p\scalar{\zbar,t-s}^q\scalar{u,\overline{t-s}}^q\scalar{\ubar,t+s}^pd\sigma(s)d\sigma(t).
	\end{align*}
	Applying the projection operator $Proj^{p,q}_{\cE^\dagger}$ (with respect to $z$) to $K_{p,q}$ results in the symplectic kernel $K_{p,q}^S$ according to 		(\ref{option1}). We thus get that
	\begin{align*}
		K_{p,q}^S(z,u)=\frac{\lambda_{p,q}}{\omega_{4n-1}\omega_{4n-3}}\int_{St^{(2)}}
		Proj^{p,q}_{\cE^\dagger}\Big(\scalar{z,\overline{t+s}}^p\scalar{\zbar,t-s}^q\Big)\scalar{u,\overline{t-s}}^q\scalar{\ubar,t+s}^pd\sigma(s)d\sigma(t).
	\end{align*}
	By performing the same computation we used to prove Lemma \ref{projC} we get that 
	$Proj^{p,q}_{\cE^\dagger}\Big(\scalar{z,\overline{t+s}}^p\scalar{\zbar,t-s}^q\Big)=g_z^{p,q}$. Applying
	the symmetric projection $Proj^{q,p}_{\cE}$ (see Remark \ref{symm}) with respect to $u$ will not change the left-hand side of the equation as by symmetry 
	$K_{p,q}^S\in\cH_{q,p}^S$ (with respect to $u$). We therefore have
	\begin{align*}
		K_{p,q}^S(z,u)&=\frac{\lambda_{p,q}}{\omega_{4n-1}\omega_{4n-3}}\int_{St^{(2)}}
		g_z^{p,q}\scalar{u,\overline{t-s}}^q\scalar{\ubar,t+s}^pd\sigma(s)d\sigma(t)\\
		&=\frac{\lambda_{p,q}}{\omega_{4n-1}\omega_{4n-3}}\int_{St^{(2)}}
		g_z^{p,q}Proj^{q,p}_{\cE}\Big(\scalar{u,\overline{t-s}}^q\scalar{\ubar,t+s}^p\Big)d\sigma(s)d\sigma(t)\\
		&=\frac{\lambda_{p,q}}{\omega_{4n-1}\omega_{4n-3}}\int_{St^{(2)}}
		g_z^{p,q}\overline{g_u^{p,q}}d\sigma(s)d\sigma(t).
	\end{align*}
\end{proof}

\end{document}